\documentclass[12pt]{amsart}
\usepackage{graphicx}
\usepackage[active]{srcltx}
\usepackage{amsthm,mathrsfs}
\usepackage{a4wide}
\usepackage{amsfonts}
\usepackage{amsmath}
\usepackage{amssymb}
\newtheorem{lemma}{Lemma}
\newtheorem{theorem}{Theorem}
\newtheorem{corollary}{Corollary}
\newtheorem{proposition}{Proposition}
\newtheorem{conjecture}{Conjecture}[section]
\newtheorem{definition}{Definition}
\newtheorem{case}{Case}

\usepackage{float}
\usepackage{url}
\usepackage{enumitem}
\usepackage{epstopdf}

\makeatletter\def\blfootnote{\xdef\@thefnmark{}\@footnotetext}\makeatother
\allowdisplaybreaks
\title[Remark on a result of Bourgain on Poissonian pair correlation]{\bf Remark on a result of Bourgain on Poissonian pair correlation}

\author{Gerhard Larcher} 
\address{ Institute of Financial Mathematics and applied Number Theory, University Linz}
\email{gerhard.larcher@jku.at}

\thanks{The author is supported by the Austrian Science Fund (FWF): Project F5507-N26, which is part of the Special Research Program "Quasi-Monte Carlo Methods: Theory and Applications" and Project  I1751-N26.}

\begin{document}

\begin{abstract}
We show for a class of sequences $\left(a_{n}\right)_{n \geq 1}$ of distinct positive integers, that for \textbf{no} $\alpha$ the sequence $\left(\left\{a_{n} \alpha\right\}\right)_{n \geq 1}$ does have Poissonian pair correlation. This class contains for example all strictly increasing integer sequences with positive upper density. This result motivates us to state a certain conjecture on Poissonian pair correlation which would be a significantly stronger version of a result of Jean Bourgain.
\end{abstract}

\date{}
\maketitle

\section{Motivation and statement of result} \label{sect_1}

\begin{definition}
Let $\left\|\cdot\right\|$ denote the distance to the nearest integer. A sequence $\left(x_{n}\right)_{n \geq 1}$ in $\left[\left.0,1\right.\right)$ is said to have (asymptotically) Poissonian pair correlations, if for each $s > 0$ the pair correlation function
$$
R_{2} \left(\left[-s,s\right], \left(x_{n}\right)_{n}, N\right):= \frac{1}{N} \# \left\{1 \leq i \neq j \leq N \left| \left\|x_{i} -x_{j} \right\|\leq \frac{s}{N} \right.\right\} 
$$
tends to $2s$ as $N \rightarrow \infty$.
\end{definition}
It is known that if a sequence $\left(x_{n}\right)_{n \geq 1}$ has Poissonian pair correlations, then it is uniformly distributed modulo 1, cf. \cite{ALT}, \cite{GL}, \cite{St}. The converse is not true in general.\\
The study of Poissonian pair correlations of sequences, especially of sequences of the form $\left(\left\{a_{n} \alpha\right\}\right)_{n \geq1}$, where $\alpha$ is an irrational, and $\left(a_{n}\right)_{n \geq 1}$ is a sequence of distinct positive integers, is primarily motivated by certain questions in quantum physics, especially in connection with the Berry-Tabor conjecture in quantum mechanics, cf. \cite{AAL}, \cite{Ma}. The investigation of Poissonian pair correlation was started by Rudnick, Sarnak and Zaharescu, cf. \cite{R1}, \cite{R2}, \cite{R3}, and was continued by many authors in the subsequent, cf. \cite{ALLB} and the references given there.\\

A quite general result which connects Poissonian pair correlation of sequences $\left(\left\{a_{n} \alpha\right\}\right)$ to concepts from additive combinatorics was given in \cite{ALLB}:\\
For a finite set $A$ of reals the {\em additive energy} $E(A)$ is defined as
$$
E\left(A\right) := \sum_{a+b=c+d} 1,
$$
where the sum is extended over all quadruples $\left(a,b,c,d\right) \in A^{4}$. Trivially one has the estimate $\left|A^{2}\right| \leq E\left(A\right)\leq \left|A\right|^{3}$, assuming that the elements of $A$ are distinct. The additive energy of sequences has been extensively studied in the additive combinatorics literature, cf. \cite{TV}. In \cite{ALLB} the following was shown:\\

\textbf{Theorem A in \cite{ALLB}.}
{\em Let $\left(a\left(n\right)\right)_{n \geq 1}$ be a sequence of distinct integers, and let $A_{N}$ denote the first $N$ elements of this sequence. If there exists a fixed $\varepsilon > 0$ such that
$$
E \left(A_{N}\right) = \mathcal{O} \left(N^{3-\varepsilon}\right),
$$
then for almost all $\alpha$ the sequence $\left(\left\{a_{n} \alpha \right\}\right)_{n \geq 1}$ has Poissonian pair correlation.}\\

On the other hand Bourgain in \cite{ALLB} showed the following negative result:\\

\textbf{Theorem B in \cite{ALLB}.}
{\em If $E\left(A_{N}\right) = \Omega \left(N^{3}\right)$ then there exists a subset of $\left[0,1\right]$ of positive measure such that for every $\alpha$ from this set the pair correlation of $\left(\left\{a_{n} \alpha\right\}\right)_{n \geq 1}$ is \textbf{not} Poissonian.}\\

In \cite{LT} the authors gave a sharper version of the result of Bourgain by showing that the set of exceptional values $\alpha$ from Theorem~2 has \textbf{full measure}.\\

In fact we conjecture that even more is true:

\begin{conjecture}
If $E \left(A_{N}\right) = \Omega \left(N^{3}\right)$ then for \textbf{every} $\alpha$ the pair correlation of $\left(\left\{a_{n} \alpha\right\}\right)_{n \geq 1}$ is \textbf{not} Poissonian.
\end{conjecture}

A very simple case of $\left(\left\{a_{n} \alpha\right\}\right)_{n \geq 1}$ with Poissonian pair correlation for \textbf{no} $\alpha$ is given by the pure Kronecker sequence $\left(\left\{n \alpha\right\}\right)_{n \geq 1}$.\\
It is the aim of this note to support this conjecture by proving it for a certain class of integer sequences $\left(a_{n}\right)_{n \geq 1}$, a class which for example also contains all strictly increasing sequences $\left(a_{n}\right)_{n \geq 1}$ with positive upper density. 

To be able to state our result we need an alternative classification of integer sequences $\left(a_{n}\right)_{n \geq 1}$ with $E \left(A_{N}\right) = \Omega \left(N^{3}\right)$:\\

For $v \in \mathbb{Z}$ let $A_{N} (v)$ denote the cardinality of the set
$$
\left\{\left(x,y\right) \in \left\{1, \ldots, N\right\}^{2}, x \neq y : a_{x} - a_{y} = v \right\}.
$$
Then
\begin{equation} \label{equ_aa}
E \left(A_{N}\right) = \Omega \left(N^{3}\right)
\end{equation}
is equivalent to
\begin{equation} \label{equ_bb}
\sum_{v \in \mathbb{Z}} A^{2}_{N} (v) = \Omega \left(N^{3}\right),
\end{equation}
which implies that there is a $\kappa > 0$ and positive integers $N_{1} < N_{2} < N_{3} < \ldots$ such that
\begin{equation} \label{equ_c22}
\sum_{v \in \mathbb{Z}} A^{2}_{N_{i}} (v) \geq \kappa N^{3}_{i}, \qquad i = 1,2,\dots.~\\
\end{equation}

It will turn out that sequences $\left(a_{n}\right)_{n \geq 1}$ satisfying \eqref{equ_aa} have a strong linear substructure. From \eqref{equ_c22} we can deduce by the Balog--Szemeredi--Gowers-Theorem (see \cite{Ba+Sz} and \cite{Gow1}) that there exist constants $c, C > 0$ depending only on $\kappa$ such that for all $i=1,2,3,\ldots$ there is a subset $A_{0}^{(i)} \subset \left(a_{n}\right)_{1 \leq n \leq N_{i}}$ such that 
$$
\left|A_{0}^{(i)}\right| \geq c N_{i} \qquad \text{and} \qquad \left|A_{0}^{(i)} + A_{0}^{(i)}\right| \leq C  \left|A_{0}^{(i)}\right| \leq C N_{i}.
$$
The converse is also true: If for all $i$ for a set $A_{0}^{(i)}$ with $A_{0}^{(i)} \subset \left(a_{n}\right)_{1 \leq n \leq N_{i}}$ with $\left| A_{0}^{(i)} \right| \geq c N_{i}$ we have $\left|A_{0}^{(i)} + A_{0}^{(i)}\right| \leq C \left|A_{0}^{(i)}\right|$, then 
$$
\sum_{v \in \mathbb{Z}} A^{2}_{N_{i}} (v) \geq \frac{1}{C} \left|A_{0}^{(i)}\right|^{3} \geq \frac{c^3}{C} N_{i}^{3}
$$
and consequently $\sum_{v \in \mathbb{Z}} A^{2}_{N} (v)= \Omega \left(N^{3}\right)$ (this an elementary fact, see for example Lemma~1~(iii) in \cite{vflev}.)\\

Consider now a subset $A_{0}^{(i)}$ of $\left(a_{n}\right)_{1 \leq n \leq N_{i}}$ with 
$$
\left|A_{0}^{(i)}\right| \geq c N_{i} \qquad \text{and} \qquad \left|A_{0}^{(i)} + A_{0}^{(i)} \right| \leq C \left|A_{0}^{(i)}\right|.
$$
By the theorem of Freiman (see \cite{Frei}) there exist constants $d$ and $K$ depending only on $c$ and $C$, i.e. depending only on $\kappa$ in our setting, such that there exists a \emph{$d$-dimensional arithmetic progression} $P_{i}$ of size at most $K N_{i}$ such that $A_{0}^{(i)} \subset P_{i}$. This means that $P_i$ is a set of the form 
\begin{equation}\label{equ_c}
P_{i} := \left\{ \left. h_{i} + \sum^{d}_{j=1} r_{j} k_{j}^{(i)} \right|  0 \leq r_{j} < s_{j}^{(i)} \right\},  
\end{equation}
with $b_{i}, k_{1}^{(i)}, \ldots, k^{(i)}_{d}, s_{1}^{(i)}, \ldots, s_{d}^{(i)} \in \mathbb{Z}$ and such that $s_{1}^{(i)} s_{2}^{(i)} \ldots s_{d}^{(i)} \leq K N_{i}$.\\

In the other direction again it is easy to see that for any set $A_{0}^{(i)}$ of the form \eqref{equ_c} we have
$$
\left|A_{0}^{(i)} + A_{0}^{(i)}\right| \leq 2^{d} K N_{i}.
$$

Based on these observations we make the following definition:

\begin{definition} \label{def_a}
Let $\left(a_{n}\right)_{n \geq 1}$ be a strictly increasing sequence of positive integers. We call this sequence {\em quasi-arithmetic of degree} $\mathbf{d}$, where $d$ is a positive integer, if there exist constants $C,K > 0$ and a strictly increasing sequence $\left(N_{i}\right)_{i \geq 1}$ of positive integers such that for all $i \geq 1$ there is a subset $A^{(i)} \subset \left(a_{n}\right)_{1 \leq n \leq N_{i}}$ with $\left|A^{(i)}\right| \geq C N_{i}$ such that $A^{(i)}$ is contained in a $d$-dimensional arithmetic progression $P^{(i)}$ of size at most $K N_{i}$.~\\
\end{definition}

The above considerations show:

\begin{proposition}
For a strictly increasing sequence $\left(a_{n}\right)_{n\geq 1}$ of positive integers we have $E\left(A_{N}\right)= \Omega \left(N^{3}\right)$ if and only if $\left(a_{n}\right)_{n \geq 1}$ is quasi-arithmetic of some degree d.
\end{proposition}

Hence our conjecture stated above is equivalent to.

\begin{conjecture}
If $\left(a_{n}\right)_{n \geq 1}$ is quasi-arithmetic of some degree d, then there is \textbf{no} $\alpha$ such that the pair correlation of $\left(\left\{a_{n} \alpha\right\}\right)_{n\geq 1}$ is Poissonian.
\end{conjecture}

Now we can state our result:

\begin{theorem} \label{th_quasi}
If $\left(a_{n}\right)_{n \geq 1}$ is quasi-arithmetic of degree $d=1$, then there is \textbf{no} $\alpha$ such that the pair correlation of $\left(\left\{a_{n} \alpha\right\}\right)_{n \geq 1}$ is Poissonian.
\end{theorem}

A simple example of quasi-arithmetic $\left(a_{n}\right)_{n \geq 1}$ of degree 1 are strictly increasing sequences of integers with positive upper density. Hence we have

\begin{corollary}
If $\left(a_{n}\right)_{n \geq 1}$ is a strictly increasing sequence of positive integers with positive upper density, i.e., 
$$
\underset{n \rightarrow \infty}{\limsup}~\frac{n}{a_{n}} > 0,
$$
then for no $\alpha$ the pair correlation of $\left(\left\{a_{n} \alpha\right\}\right)_{n \geq 1}$ is Poissonian.
\end{corollary}

The proof of Theorem~\ref{th_quasi} will be provided in the next two sections. \\
For $d \geq 2$ it seems to be necessary to study the structure of sets of the form
$$
\left\{K_{1} \alpha_{1} +K_{2} \alpha_{2} + \ldots + K_{d} \alpha_{d} \left| 1 \leq K_{i} \leq n_{i}, i=1, \ldots, d \right.\right\}
$$
very carefully. This indeed is a not at all trivial task. See for example \cite{HM} for an excellent survey on this topic and the references given there.

\section{Auxiliary Results}
\begin{proof}
Let $\left(a_{n}\right)_{n \geq 1}$ be quasi-arithmetic of degree 1. Let $c \leq 1, K \geq 1$, the strictly-increasing sequence $\left(N_{i}\right)_{i \geq 1}$ of positive integers, and $A^{(i)} \subseteq \left(a_{n}\right)_{1 \leq n \leq N_{i}}$ with $\left|A^{(i)}\right|\geq c \cdot N_{i}$ be such that $A^{(i)}$ is contained in a one-dimensional arithmetic progression $P^{(i)}$ of size at most $KN$:\\

Let us fix some $i$, and set for simplicity\\
$N := N_{i}$\\
$A:= A^{(i)} := \left(a_{n_{j}}\right); j = 1, \ldots, W$ with $1 \leq n_{1} < n_{2} < \ldots < n_{W} \leq N$, and $W= \gamma^{(i)} \cdot N$ with $\gamma := \gamma^{(i)} \geq c$.\\
$P := P^{(i)} := \left\{h + rk \left| \right. 0 \leq r < \Gamma^{(i)} \cdot N \right\}$ with $\Gamma := \Gamma^{(i)} \leq K.$\\

In fact in the following we will consider a certain subsequence of the $N_{i}$, namely:\\
let $\tilde{\gamma} := \liminf \gamma^{(i)} \geq c$.\\
We consider only these $N_{i_{l}}$ with indices $i_{l}$ such that $\frac{\tilde{\gamma}}{2} \leq \gamma^{\left(i_{l}\right)} \leq 2 \tilde{\gamma}$. For these $i_{l}$ let $\tilde{\Gamma} := \limsup \Gamma^{\left(i_{l}\right)}$.\\

We have $\tilde{\Gamma} \leq K$. We consider only these indices with $2 \tilde{\Gamma} \geq \Gamma^{\left(i_{l}\right)} \geq\frac{\tilde{\Gamma}}{2}$. For simplicity we assume that $N_{i}$ already satisfies these conditions, consequently we can choose $\frac{\tilde{\gamma}}{2}= c$ and $2 \tilde{\Gamma} = K$ and hence $c \leq \gamma^{(i)} \leq 4 c$ and $\frac{K}{4} \leq \Gamma^{(i)} \leq K$ always.\\

Further we may assume $h=0$. For general $h$ the proof runs quite analogously. And we may assume $k=1$, since studying $\left\{r k \alpha\right\}$ is nothing else than studying $\left\{r \alpha'\right\}$ with $\alpha' = k \alpha$. So $P = \left\{r \left| \right. 0 \leq r < \Gamma \cdot N \right\}$.\\

Let $\alpha$ have continued fraction expansion $\alpha = \left[0; \alpha_{1}, \alpha_{2}, \ldots \right]$ and best approximation denominators $\left(q_{i}\right)_{i \geq 0}$ with $q_{i+1} = \alpha_{i+1} q_{i} + q_{i-1}$. Set $M:= \Gamma \cdot N$, and let $l$ be such that $q_{l} \leq M < q_{l+1}$ and $b$ with $1 \leq b < \alpha_{l+1}$ be such that $b q_{l} \leq M < \left(b+1\right) q_{l}$. For simplicity in the following we set $q:= q_{l}, a = \alpha_{l+1}$. We frequently will use 
\begin{equation} \label{equ_uno}
\frac{bq}{K} \leq \frac{bq}{\Gamma} \leq N = \frac{M}{\Gamma} \leq \frac{\left(b+1\right)q}{\Gamma} \leq \frac{2 bq}{\Gamma} \leq \frac{8 bq}{K}.
\end{equation}
Let $S_{\alpha} := \left(\left\{a_{n_{j}} \alpha \right\}\right)_{j=1, \ldots, W}$ and $\bar{S}_{\alpha} := \left(\left\{j \alpha \right\}\right)_{j=1, \ldots, M}$. Then $S_{\alpha}$ is a subset of $\bar{S}_{\alpha}$ with $W = \gamma \cdot N = \frac{\gamma}{\Gamma} \cdot M$ elements. We order the elements of $\bar{S}_{\alpha}$ in $\left[\left.0,1\right.\right)$ in ascending order, i.e., $\bar{S}_{\alpha} = \left\{\beta_{1}, \ldots, \beta_{M}\right\}$ with $0 < \beta_{1} < \beta_{2} < \ldots < \beta_{M} < 1$.\\

Let us further assume that the index $l$ which is defined by $q=q_{l} \leq M < q_{l+1}$ is even (for $l$ odd we argue quite similar). Then it is well known that $\bar{S}_{\alpha}$ consists of $q$ bundles, each bundle consisting of $b$ or $b+1$ elements, and each bundle contained in exactly one of the intervals $\left[\left.\frac{i}{q}, \frac{i+1}{q}\right.\right), i=0, \ldots, q-1$, as is sketched in Figure 1.
\begin{center}
\includegraphics[angle=0,width=150mm]{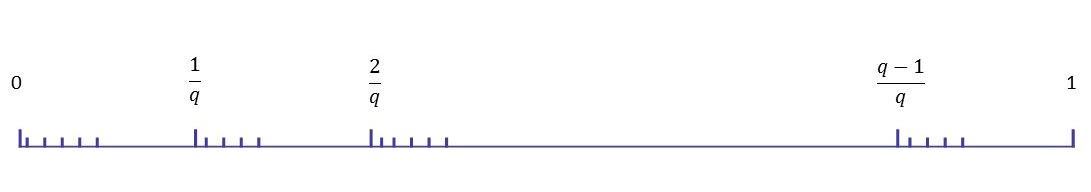}\\
Figure 1
\label{fig_1}
\end{center}

Let $\beta_{u} < \beta_{u+1} < \ldots < \beta_{u+w}$ with $w = b-1$ or $w=b$ be the elements of $\bar{S}_{\alpha}$ contained in $\left[\left.\frac{i}{q}, \frac{i+1}{q}\right.\right)$. See Figure 2.\\

\begin{center}
\includegraphics[angle=0,width=150mm]{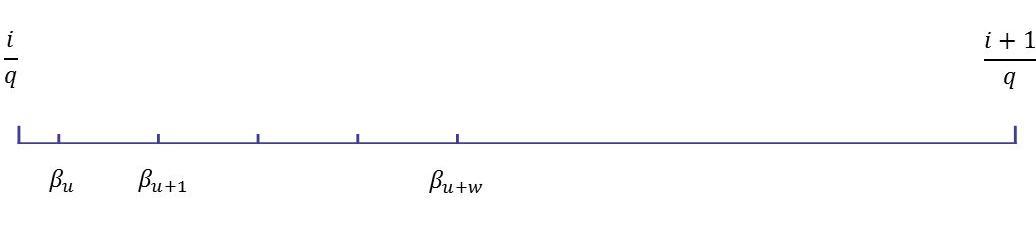}\\
Figure 2
\label{fig_2}
\end{center}

For simplicity in all the following we assume $w = b$, the case $w=b-1$ is treated in the same way. Then by basic properties of continued fractions we know that
$$
\delta := \beta_{u+1} - \beta_{u} = \beta_{u+2} - \beta_{u+1} = \ldots = \beta_{u+w} - \beta_{u+w-1},
$$
with
\begin{equation} \label{equ_q}
\frac{1}{3 q a} \leq \frac{1}{q \left(a+2\right)} < \delta < \frac{1}{qa}
\end{equation}
and hence
\begin{equation} \label{equ_b}
\frac{b}{3 q a} < \beta_{u+w}-\beta_{u} < \frac{b}{q a}.
\end{equation}
Further
\begin{equation} \label{equ_a}
\beta_{u} - \frac{i}{q} < \delta < \frac{1}{q a}.
\end{equation}
Let $P_{i}$ for $i= 1, \ldots, b$ be the subset of $\bar{S}_{\alpha}$ consisting of the $i$-th elements $\left(\beta_{u+i-1}\right)$ from each bundle. For a fixed $i$ we denote the elements of $P_{i}$ in ascending order by $v_{0}~<~v_{1}~<~\ldots~<~v_{q-1}$. Note, $v_{p}$ is contained in $\left[ \left. \frac{p}{q}, \frac{p+1}{q} \right.\right)$. Again by basic properties of continued fractions we always have
\begin{equation} \label{equ_dos}
v_{j+1} - v_{j} \geq \frac{1}{2 q}
\end{equation}
for all $j$.\\

For a fixed positive integer $m < q$ consider now the set of distances $v_{m}- v_{0}, v_{m+1} - v_{1}, \ldots, v_{q-1}-v_{q-m-1}$. We fix a $j$ and consider $v_{m+j}-v_{j}$. $v_{j}$ is given by $v_{j} = \left\{\left(\left(i-1\right)\cdot q +y\right)\alpha\right\}$ for some $y \in \left\{0,1,\ldots, q-1\right\}$ and (with the same $y$) $v_{m+j}$ is given by \\
$v_{m+j} = \left\{\left(\left(i-1\right)q + \left(y+mq'\right)~\text{mod}~q\right)\alpha\right\}$ where $q' := q_{l-1}$. \\
Hence $v_{m+j}-v_{j} = \left\{\left(y-\left(y+m q'\right)~\text{mod}~ q\right) \alpha \right\}$, and this is either $\left\{\left(- \left(m q'\right) \text{mod}~ q\right) \alpha \right\}$ or $\left\{\left(q- \left(m q'\right) \text{mod}~ q\right)\alpha\right\}$. Hence, the set of distances 
\begin{equation} \label{equ_cinco}
v_{m} - v_{0}, v_{m+1} - v_{1}, \ldots, v_{q-1} - v_{q-m-1}
\end{equation}
can attain a most two different values.\\

For the proof we will need the following three simple Lemmata:
\begin{lemma} \label{lem_uno}
Let $\mathcal{B} = \left[\left.0,B\right.\right)$ be an interval of length $B \leq 1$. Let $0 < \tau \leq \frac{B}{2}$. For an integer $\sigma \geq 2$ let $x_{1}, x_{2}, \ldots, x_{L}$ with $L = \left(\frac{B}{\tau}+1\right) \sigma$ be $L$ points in $\mathcal{B}$. (We assume for simplicity $\frac{B}{\tau} \in \mathbb{N}$.) Then
$$
\wedge := \# \left\{1 \leq i \neq j \leq L \left| \right. \left|x_{i}-x_{j}\right|<\tau\right\} \geq \frac{\sigma^{2}B}{2 \tau}.
$$
\end{lemma}
\begin{proof}
It is obvious that $\wedge$ becomes minimal if the $L$ points are distributed in the following way:\\
$\sigma$ times the point $j \cdot \tau$ for $j=0,1,\ldots, \frac{B}{\tau}$. For this distribution we have \\
$\wedge = \sigma \left(\sigma-1\right)~\cdot~\frac{B}{\tau}~\geq~\frac{\sigma^{2}B}{2 \tau}$.
\end{proof}
\begin{lemma} \label{lem_dos}
Let $\mathcal{B}_{0}, \mathcal{B}_{1}, \ldots, \mathcal{B}_{q-1}$ be $q$ intervals of length $B$ each. Let $0 < \tau \leq \frac{B}{2}$. Assume there are $L_{i}$ points $x^{(i)}_{1}, \ldots, x^{(i)}_{L_{i}}$ in $\mathcal{B}_{i}$ for $i=0,\ldots, q-1$ with $L_{0} + \ldots + L_{q-1} = \left(\frac{B}{\tau} + 1\right) \cdot q \cdot \psi$, with $\psi \geq 3$. Then
$$
\tilde{\wedge} := \sum^{q-1}_{i=0} \# \left\{1 \leq k \neq l \leq L_{i}: \left|x^{(i)}_{k} - x^{(i)}_{l} \right| < \tau\right\} \geq \frac{B q}{2 \tau} \left(\psi -2\right)^{2}.
$$
\end{lemma}
\begin{proof}
Let $L_{i} := \left(\frac{B}{\tau} +1\right) \cdot \tilde{\sigma}_{i}$. Let $\sigma_{i} := \begin{cases} 0 & \text{if}~\tilde{\sigma}_{i} < 2 \\ \left[\tilde{\sigma}_{i}\right] & \text{if}~\tilde{\sigma}_{i} \geq 2 \end{cases}$\\
and let $L'_{i} := \left(\frac{B}{\tau}+1\right)\cdot \sigma_{i}$. Then $L_{1}' + \ldots + L_{q}' \geq \left(\frac{B}{\tau}+1\right) q \cdot \left(\psi -2\right)$, that is $\sigma_{1} + \ldots + \sigma_{q} \geq q \cdot \left(\psi -2\right)$. By Lemma~\ref{lem_uno} we have
$$
\tilde{\wedge} \geq \sum^{q}_{i=1} \frac{\sigma^{2}_{i} \cdot B}{2 \tau} \geq \frac{B}{2 \tau} \sum^{q}_{i=1} \left(\psi-2\right)^{2} = \frac{B q}{2 \tau} \cdot \left(\psi-2\right)^{2}.
$$
\end{proof}
\begin{lemma} \label{lem_tres}
Let $\alpha \in \left(0,1\right)$ be given. For A (large enough) let $\Delta_{1}, \ldots, \Delta_{A \alpha-1}$ be $A \alpha-1$ positive integers with $\sum^{\alpha A-1}_{i=1} \Delta_{i} \leq A$. Then there exist at least $\frac{\alpha^{2} A}{4}$ elements $\Delta_{i}$ attaining the same value $\Delta$ with $\Delta \leq \frac{4}{\alpha^{2}}$.
\end{lemma}
\begin{proof}
This is an easy exercise and is left to the reader.
\phantom\qedhere
\end{proof}

\section{Proof of the Theorem}

Before we proceed, we repeat: We study the pair correlation of $\left(\left\{a_{n} \alpha\right\}\right)_{n=1,\ldots, N}$. We have
$$
S_{\alpha} = \left(\left\{a_{n_{j}} \alpha\right\}\right)_{j=1,\ldots, W} \subseteq  \left(\left\{a_{n} \alpha\right\}\right)_{n=1,\ldots, N}
$$
and
$$
S_{\alpha} \subseteq \bar{S}_{\alpha} = \left(\left\{j \alpha\right\}\right)_{j=1,\ldots, M}.
$$
Further $W = \gamma \cdot N$ and $M = \Gamma \cdot N$. By $L_{p}$ we denote the number of elements of $S_{\alpha}$ lying in the interval $\left[\left.\frac{p}{q}, \frac{p+1}{q}\right.\right)$. Note that $L_{0} + \ldots + L_{q-1} = W$.\\

Now we distinguish 4 cases. (Indeed, the 4 cases cover all possible situations.)\\

\begin{case}
\begin{equation} \label{equ_tres}
b \delta \geq \frac{2^{7}}{c} \cdot \frac{1}{N}
\end{equation}
and
\begin{equation} \label{equ_cuatro}
\frac{a}{b} \geq \frac{2^{23}}{c^{2}}.
\end{equation}
This means: The length of the bundle of points from $\bar{S}_{\alpha}$ in an interval $\left[ \left. \frac{p}{q}, \frac{p+1}{q}\right.\right)$ is significantly larger than $\frac{1}{N}$ and $b$ is significantly smaller than $a$.
\end{case}
Then we choose $s:= \frac{2^{6}}{c}$. By \eqref{equ_tres} we have $\frac{s}{N} \leq \frac{b \delta}{2}$. Let $\mathcal{B}_{p}$ be the smallest interval containing the bundle of points from $\bar{S}_{\alpha}$ in an interval $\left[\left. \frac{p}{q}, \frac{p+1}{q}\right.\right)$. All $\mathcal{B}_{p}$ have the same length $B:= b \delta$. We count pairs of points from $S_{\alpha}$ with distance at most $\tau := \frac{s}{N}$. We have 
\begin{equation} \label{equ_seis}
\tau \leq \frac{B}{2}.
\end{equation}
Let $L_{p}$ denote the number of points from $S_{\alpha}$ contained in $\mathcal{B}_{p}$. Then by \eqref{equ_uno}, \eqref{equ_seis} and \eqref{equ_cuatro} we have:
\begin{eqnarray*}
L_{0} + \ldots + L_{q-1}  & = & W  \geq c  N \geq c \frac{b q}{K} = \frac{B}{\tau}  q  \left(\frac{\tau}{B}  \frac{c b}{K}\right) =\\
& = & \frac{B}{\tau} q \left(\frac{s}{N b \delta } \frac{c b}{K}\right)\geq \frac{B}{\tau} q \left(\frac{s a q}{\frac{8 b q}{K}} \frac{c}{K}\right) =\\
& = & \frac{B}{\tau} q \left(s \frac{a}{b} \frac{c}{8}\right) = \frac{B}{\tau} q \left(8 \frac{a}{b}\right) \geq \left(\frac{B}{\tau} + 1\right) q \left(\frac{1}{4} \frac{a}{b}\right).
\end{eqnarray*}
So we can apply Lemma~\ref{lem_dos} with $\psi = \frac{1}{4} \frac{a}{b} \geq \frac{2^{21}}{c^{2}} \geq 4$, and we obtain (using \eqref{equ_q}, and \eqref{equ_cuatro})
\begin{eqnarray*}
R_{N} (s) & := & \frac{1}{N} \cdot \# \left\{1 \leq i \neq j \leq N \left| \left\|\left\{a_{i} \alpha\right\}-\left\{a_{j} \alpha\right\}\right\| \leq \frac{s}{N}\right.\right\} \geq \\
& \geq & \frac{1}{N} \cdot \# \left\{1 \leq i \neq j \leq W \left| \left\|\left\{a_{n_{i}} \alpha\right\} - \left\{a_{n_{j}} \alpha \right\}\right\|\leq \frac{s}{N}\right.\right\} \geq \\
& \geq & \frac{1}{N} \frac{B q}{2 \tau} \left(\psi -2\right)^{2} \geq \frac{ b \delta q}{2 s} \left(\frac{a}{8 b}\right)^{2} \geq \\
& \geq & \frac{b}{6 a s} \left(\frac{a}{8 b}\right)^{2} = s \frac{1}{s^{2}} \frac{a}{b} \frac{1}{6 \cdot 2^{6}} = \\
& = & s \frac{c^{2}}{2^{12}} \frac{1}{6 \cdot 2^{6}} \frac{a}{b} \geq 4 s.
\end{eqnarray*}
\begin{case}
\begin{equation} \label{equ_siete}
\frac{a}{b} \leq \frac{2^{23}}{c^{2}}
\end{equation}
and
\begin{equation} \label{equ_ocho}
\frac{1}{N} < \frac{c^{5}}{K \cdot 2^{23}} \cdot \frac{1}{q}
\end{equation}
\end{case}
Condition \eqref{equ_siete} implies
\begin{equation} \label{equ_nueve}
b \geq \frac{c^{2}}{2^{23}} a
\end{equation}
and \eqref{equ_ocho} by \eqref{equ_uno} implies
\begin{equation} \label{equ_diez}
b > \frac{K^{2} \cdot 2^{26}}{c^{5}}.
\end{equation}
There exist at least $\frac{q c}{2 K}$ intervals $\left[\left. \frac{p}{q}, \frac{p+1}{q}\right.\right)$ which contain at least $\frac{b c}{2K}$ points from $S_{\alpha}$.\\
(Otherwise we had $\left|S_{\alpha}\right| \leq \frac{q c}{2 K} \cdot b + \left(q - \frac{q c}{2 K}\right) \cdot \frac{bc}{2 K} < \frac{c b q}{K} \leq c \cdot N \leq W$, a contradiction.)\\
We denote the elements of $\left[ \left. \frac{p}{q}, \frac{p+1}{q}\right.\right) \cap S_{\alpha}$ with $\beta_{1}^{\left(p\right)} < \ldots < \beta_{W_{p}}^{\left(p\right)}$. We consider the set of ``normalized'' differences
$$
\frac{\beta^{\left(p\right)}_{j+1} - \beta^{\left(p\right)}_{j}}{\delta} \quad\quad \text{for}~j=1, \ldots, W_{p}-1, ~\text{and}~p=1, \ldots, q.
$$
These values are positive integers, we denote them by $\gamma_{1}, \ldots, \gamma_{Q}$, where $Q = W - q \geq \frac{c b q}{K} - q > \frac{c b q}{2 K}$ because of \eqref{equ_diez}. Further $\gamma_{1} + \ldots + \gamma_{Q} \leq b q$. By Lemma~\ref{lem_tres} therefore at least $\frac{c^{2}}{16 K^{2}} bq$ of the $\gamma_{i}$ must have a same value, say $\beta$, with
\begin{equation} \label{equ_once}
\beta \leq \frac{16 K^{2}}{c}.
\end{equation}
That implies: There exist at least 
\begin{equation} \label{equ_diecinueve}
\frac{c^{2}}{16 K^{2}} b q
\end{equation}
pairs of points of $S_{\alpha}$ with distance exactly $\beta \cdot \delta$. By \eqref{equ_once}, \eqref{equ_uno}, and \eqref{equ_nueve} we have
\begin{eqnarray*}
\beta \delta & \leq & \frac{16 K^{2}}{c} \delta \leq \frac{16 K^{2}}{c} \frac{1}{a q} = \frac{16 K^{2}}{c} \frac{1}{aq} N \frac{1}{N} \leq \\
& \leq & \frac{16 K^{2}}{c} \frac{1}{a q} \frac{8 b q}{K} \frac{1}{N} = \frac{2^{7}K}{c} \frac{b}{a} \frac{1}{N} \leq \frac{2^{7}K}{c} \frac{1}{N}
\end{eqnarray*}
and
\begin{eqnarray*}
\beta \delta & \geq & \delta \geq \frac{1}{3 a q} = \frac{N}{3 a q} \frac{1}{N} \geq \frac{b q}{3 a q K} \frac{1}{N} \geq \\
& \geq & \frac{1}{3 K} \frac{b}{a} \frac{1}{N} \geq \frac{c^{2}}{3 \cdot 2^{23} K} \frac{1}{N}.
\end{eqnarray*}
We choose now $s_{1} < s_{2}$ with $s_{2} - s_{1}$ very small such that $\frac{s_{1}}{N} < \beta \delta < \frac{s_{2}}{N}$. Thereby it will be crucial that $s_{1}, s_{2}$ are chosen from a finite set $\mathcal{D}$ which is defined depending on the ``universal'' constants $c, K$ only. We define $\mathcal{D}$ first:
$$
\mathcal{D} := \left\{\frac{c^{2}}{3 \cdot 2^{23} \cdot K} + j \cdot \frac{c^{2}}{2^{9} \cdot K} \left| j =0,1, \ldots, \frac{2^{16} \cdot K^{2}}{c^{3}} \right.\right\}.
$$
Then we find $s_{1}, s_{2} \in \mathcal{D}$ with $\frac{s_{1}}{N} < \beta\delta < \frac{s_{2}}{N}$, and 
\begin{equation} \label{equ_doce}
s_{2} = s_{1} + \frac{c^{2}}{2^{9} \cdot K}.
\end{equation}
Then by \eqref{equ_diecinueve} and \eqref{equ_doce} we have
$$
R_{N} \left(s_{2}\right)- R_{N} \left(s_{1}\right) \geq \frac{1}{N} \frac{c^{2}}{16 K^{2}} b q \geq \frac{K}{8 b q} \frac{c^{2}}{16 K^{2}} bq = \frac{c^{2}}{2^{7} K} = 4 \left(s_{2} - s_{1}\right).
$$
(Note, that $R_{N} \left(s_{2}\right) - R_{N} \left(s_{1}\right)$ should be approximately $2 \left(s_{2} - s_{1}\right) !$)
\begin{case} 
\begin{equation} \label{equ_trece}
b \delta \leq \frac{2^{7}}{c} \frac{1}{N}
\end{equation}
and
\begin{equation} \label{equ_catorce}
\frac{1}{N} < \frac{c^{5}}{K 2^{29}} \frac{1}{q}
\end{equation}
\end{case}
In this case we choose $s=1$. That is: We consider distances of points less or equal $\frac{1}{N}$, where this bound $\frac{1}{N}$ is significantly smaller than $\frac{1}{q}$, but is of order of the length of an interval $b \delta$ or larger. Condition \eqref{equ_catorce} and formula \eqref{equ_uno} imply $\frac{K}{8 b q} < \frac{1}{N} < \frac{c^{5}}{K \cdot 2^{29}} \frac{1}{q}$, hence
\begin{equation} \label{equ_quince}
b > \frac{2^{26}}{c^{5}} K^{L}.
\end{equation}
We recall that for every $p=0,\ldots, q-1$ we denote with $L_{p}$ the number of points of $S_{\alpha}$ in the interval $\mathcal{B}_{p}$. We have by \eqref{equ_uno}:
$$
L_{0} + L_{1} + \ldots + L_{q-1} = W \geq c N \geq c \frac{b q}{K}.
$$
Since $L_{p} \leq b$ always, we can conclude that at least $\frac{c}{2 K} q$ of the $L_{p}$ are at least $\frac{c}{2 K} b$. (This follows from $\frac{c}{2 K} q \cdot b + \left(q - \frac{c}{2 K} q\right) \cdot \frac{c}{2 K} b = c \cdot \frac{b q}{K} \cdot \left(\frac{1}{2} + \frac{1}{2} \left(1 - \frac{c}{2 K}\right)\right) < c \cdot \frac{b q}{K}\leq W.$)\\
We denote the at least $\frac{c}{2 K} q$ integers $p$ for which the $L_{p}$ are at least $\frac{c}{2 K} b$ by $p'$. We have $\sum_{p'} L p' \geq \frac{c^{2}}{4 K^{2}} bq$.

The interval $\mathcal{B}_{p'}$ has length $B = b \delta$. We divide $\mathcal{B}_{p'}$ in intervals of length $\frac{1}{N}$. The number of these intervals
$$
\left\lceil \frac{B}{\frac{1}{N}}\right\rceil = \left\lceil \frac{b \delta}{\frac{1}{N}}\right\rceil \leq \frac{2^{8}}{c} < \frac{1}{4} \cdot \frac{c}{2 K} b < \frac{1}{4} L_{p'}.
$$
(Here we used \eqref{equ_trece} and \eqref{equ_quince}.) Therefore each $\mathcal{B}_{p'}$ contains an interval I of length $\frac{1}{N}$ which contains at least $\left[\frac{L_{p'}}{2^{8}} \cdot c\right]$ elements, which is at least 4, hence
$$
\left[\frac{L_{p'}}{2^{8}} c\right] \geq \frac{L_{p'}}{\sqrt{2} \cdot 2^{8}} c \geq 2.
$$ 
This gives us
\begin{eqnarray*}
R_{N} (s) & \geq & \frac{1}{N} \sum_{p'} \left(\frac{L_{p'}}{\sqrt{2}\cdot 2^{8}} c\right) \left(\frac{L_{p'}}{\sqrt{2} \cdot 2^{8}} c -1\right) \geq \\
& \geq & \frac{1}{2 N} \sum_{p'} \left(\frac{L_{p'} c}{\sqrt{2}\cdot 2^{8}}\right)^{2} =\\
& = & \frac{1}{2 N} \frac{c^{2}}{2^{17}} \sum_{p'} \left(L_{p'}\right)^{2} \geq\\
& \geq & \frac{1}{2N} \frac{c^{2}}{2^{17}} \frac{c}{2K} q \left(\frac{c}{2K} b\right)^{2} =\\
& = & \frac{c^{5}}{2^{21}K^{3}} \frac{q b^{2}}{N} \geq \frac{c^{5}}{2^{24}K^{2}} b > 4
\end{eqnarray*}
by \eqref{equ_quince}. (Note, that this quantity should approach $2s=2$ !)\\

It remains to handle
\begin{case}
\begin{equation} \label{equ_dieciseis}
\frac{1}{N} \geq \frac{c^{5}}{K \cdot 2^{29}} \frac{1}{q}
\end{equation}
\end{case}
This condition because of \eqref{equ_uno} implies $b \leq \frac{K^{2} \cdot 2^{29}}{c^{5}}$. We recall the notion $P_{i}$ for $i=1,\ldots, b$ for the subset of $\bar{S_{\alpha}}$ consisting of the $i$-th elements from each bundle of $b$ points of $\bar{S_{\alpha}}$ from Section 2. For fixed $i$ we denoted the elements of $P_{i}$ by $v_{0} < v_{1} < \ldots < v_{q-1}$. At least one of these $b$ sets $P_{i}$, say $P$, contains at least $\frac{\gamma N}{b} \geq c \frac{q}{K}$ elements from $S_{\alpha}$. Note, that $S_{\alpha}$ has $W= \gamma  N \geq c  N$ elements. For simplicity of notation we assume that $c \frac{q}{K}$ is an integer in the following. Let $u_{1} < u_{2} < \ldots < u_{\frac{c q}{K}}$ be $\frac{c q}{K}$ elements from $S_{\alpha}\cap P$. Let $u_{i} \in \left[\left.\frac{p_{i}}{q}, \frac{p_{i}+1}{q}\right.\right)$, then $0 \leq p_{1} < p_{2} < \ldots < p_{\frac{c q}{K}} \leq q-1$, and by \eqref{equ_dos} we have $u_{i+1}-u_{i} \geq \frac{1}{2 q}$ always.\\

We consider the differences
$$
\Delta_{i} := p_{i+1} - p_{i} \quad \quad \text{for}~i=1, \ldots, \frac{c q}{K}-1.
$$
These are $\frac{cq}{K}-1$ positive integers with
$$
\sum^{\frac{cq}{K}-1}_{i=1} \Delta_{i} = p_{\frac{cq}{K}} - p_{1} \leq q-1.
$$
Hence -- by Lemma~\ref{lem_tres} -- at least $\frac{c^{2}q}{4 K^{2}}$ (for $q$ large enough) of them must have the same value $p \leq \frac{4 K^{2}}{c^{2}}$. Choose $\frac{c^{2}q}{4 K^{2}}$ such $\Delta_{i}$ (for simplicity of notation we assume $\frac{c^{2}q}{8 K^{2}}$ to be an integer) with the same value $p \leq \frac{4 K^{2}}{c^{2}}$, say 
$$
\Delta_{i_{1}}, \Delta_{i_{2}}, \ldots, \Delta_{i_{\frac{c^{2}q}{4 K^{2}}}} \text{with}~1 \leq i_{1} < i_{2} < \ldots < i_{\frac{c^{2}q}{4 K^{2}}} \leq \frac{c q}{K}-1.
$$ 
Consider the distances of the corresponding points of $S_{\alpha} \cap P$, i.e.,
$$
u_{i_{j}+1} - u_{i_{j}} \quad \quad \text{for}~j=1, \ldots, \frac{c^{2}q}{4 K^{2}}.
$$
These distances by \eqref{equ_cinco} attain at most two different values. Hence at least 
\begin{equation} \label{equ_one}
\frac{c^{2}q}{8 K^{2}}
\end{equation}
of $u_{i_{j}+1} - u_{i_{j}}$ have the same value, say $\kappa$, with
$$
\kappa = \frac{p}{q} + \frac{\tau}{a q},
$$ 
with a fixed $p$ with $1 \leq p \leq \frac{4 K^{2}}{c^{2}}$, and a real $\tau$ with $\left|\tau\right| \leq 1$. Moreover (see \eqref{equ_dos}) we have $\kappa \geq \frac{1}{2 q}$. So (by \eqref{equ_uno} and \eqref{equ_dieciseis})
\begin{equation} \label{equ_diecisiete}
\frac{1}{2 K} \frac{1}{N} \leq \frac{1}{2 q} \leq \kappa \leq \frac{2 K}{q} + \frac{1}{a q} \leq \frac{4 K^{2}}{c^{2}q} \leq \frac{1}{N} \frac{2^{32} \cdot K^{2} \left(4 K^{2}\right)}{c^{7}}.
\end{equation}
We choose now $s_{1} < s_{2}$ with $s_{2} -s_{1}$ very small such that $\frac{s_{1}}{N} < \kappa < \frac{s_{2}}{N}$. Thereby it will be crucial that $s_{1}, s_{2}$ are chosen from a finite set $\mathcal{E}$ which is defined depending on the ``universal'' constants $c,K$ only. We define $\mathcal{E}$ first:
$$
\mathcal{E} := \left\{\frac{1}{2 K} + j \cdot \frac{c^{7}}{K^{3} \cdot 2^{34}} \left| j=0, \ldots, \left\lceil \frac{\left(4 K^{2}\right) K^{5}\cdot 2^{66}}{c^{14}}\right\rceil\right. \right\}.
$$
Then we find $s_{1}, s_{2} \in \mathcal{E}$ with
$$
\frac{s_{1}}{N} < \kappa < \frac{s_{2}}{N}
$$
and
\begin{equation} \label{equ_dieciocho}
s_{2} = s_{1} + \frac{c^{7}}{K^{3} \cdot 2^{34}}.
\end{equation}
Then we have by \eqref{equ_dieciseis}, \eqref{equ_one} and \eqref{equ_dieciocho}:  
$$
R_{N} \left(s_{2}\right) - R_{N} \left(s_{1}\right) \geq \frac{1}{N} \frac{c^{2}q}{8 K^{2}} \geq \frac{c^{7}}{K^{3}\cdot 2^{32}} = 4 \left(s_{2}-s_{1}\right).
$$
(Note, that $R_{N}\left(s_{2}\right)-R_{N}\left(s_{1}\right)$ should be approximately $2 \left(s_{2}-s_{1}\right)$ !)\\

Putting everything together: We can find some $s$, either $s=\frac{2^{6}}{c}$, or $s=1$, or $s$ from $\mathcal{D}$ or $s$ from $\mathcal{E}$ such that 
$$
\lim_{N \rightarrow \infty} R_{N} \left(s\right) = 2s
$$
cannot hold.
\section*{Acknowledgement}
Many thanks to Sam Chow who made me aware of a gap in the first version of the proof.
 
\end{proof}

\end{document}